\documentclass[11pt]{amsart}
\usepackage[T1]{fontenc}
\usepackage{amsfonts}
\usepackage{verbatim}
\usepackage{amsmath, amsthm, amssymb, enumerate}
\usepackage[comma,sort&compress]{natbib}
\usepackage{rotating}
\usepackage{subfigure}

\numberwithin{equation}{section}
\allowdisplaybreaks 

\usepackage{graphicx,subfigure}

\allowdisplaybreaks 
\numberwithin{equation}{section}
\newtheorem{lemma}{Lemma}[section]

\newtheorem{prop}{Proposition}

\newtheorem{theorem}[prop]{Theorem}

\newtheorem{corollary}[prop]{Corollary}
\newtheorem{remark}[prop]{Remark}

\theoremstyle{definition}

\newcommand{\reals}{{\mathbb R}}
\newcommand{\bbr}{\reals}

        \newcommand{\vrightarrow}{\,{\buildrel v \over \rightarrow}\,}

\newcommand{\CurveC}{{\mathcal C}}

\newcommand{\one}{{\bf 1}}

\newcommand{\alphain}{\alpha_{\text in}}
\newcommand{\alphaout}{\alpha_{\text out}}
\newcommand{\deltain}{\delta_{\text in}}
\newcommand{\deltaout}{\delta_{\text out}}

\newcommand\independent{\protect\mathpalette{\protect\independenT}{\perp}}
\def\independenT#1#2{\mathrel{\rlap{$#1#2$}\mkern2mu{#1#2}}}

\def\E{\mathbb{E}}
\def\bzero{\boldsymbol 0}

\usepackage{pdfsync, url}

\begin{document}

\title[Regular variation of in- and out-degrees]{Nonstandard regular variation of
 in-degree and  out-degree in the preferential attachment model}

\author[G. Samorodnitsky]{Gennady Samorodnitsky}
\address{School of Operations Research and Information Engineering\\
and Department of Statistical Science \\
Cornell University \\
Ithaca, NY 14853}
\email{gs18@cornell.edu}
\author[S. Resnick]{Sidney Resnick}
\address{School of Operations Research and Information Engineering\\
Cornell University \\
Ithaca, NY 14853}
\email{sir1@cornell.edu}
\author[D. Towsley]{Don Towsley}
\address{Department of  Computer Science \\
University of Massachusetts \\
Amherst, MA 01003}
\email{towsley@cs.umass.edu}
\author[R. Davis]{Richard Davis}
\address{Department of Statistics\\
Columbia University\\ 
New York, NY 10027}
\email{rdavis@stat.columbia.edu}
\author[A. Willis]{Amy Willis}
\address{Department of Statistical Science\\
Cornell University \\
Ithaca, NY 14853}
\email{adw96@cornell.edu}
\author[P. Wan]{Phyllis Wan}
\address{Department of Statistics\\
Columbia University\\ 
New York, NY 10027}
\email{phyllis@stat.columbia.edu}

\thanks{The research of the authors was  supported by MURI ARO
Grant  W911NF-12-10385  to Cornell University.}

\keywords{multivariate heavy tails, preferential attachment model,
scale free networks.
\vspace{.5ex}}

\setcounter{lemma}{0}


\begin{abstract}
For   the directed edge preferential attachment network growth model
studied by \cite{bollobas:borgs:chayes:riordan:2003} and
\cite{krapivsky:redner:2001}, we prove that the joint distribution of
in-degree and 
out-degree 
 has  jointly regularly varying
tails. 
Typically the marginal tails of the  in-degree distribution and the out-degree
distribution have different regular variation indices and so the joint
regular variation is non-standard.
Only marginal regular variation has been
previously established for this distribution in the cases where the
marginal tail indices are different. 
\end{abstract} 

\maketitle
\section{Introduction}\label{sec: Intro}

The directed edge  preferential attachment model 
studied by \cite{bollobas:borgs:chayes:riordan:2003} and 
\cite{krapivsky:redner:2001}
is a model for a growing directed random graph. The dynamics of the
model are as follows. Choose as parameters nonnegative real numbers
$\alpha, 
\beta, \gamma$, $\delta_{\text in}$ and $\delta_{\text out}$, such
that $\alpha+\beta+\gamma=1$. To avoid degenerate situations we will
assume  that each of the numbers $\alpha, 
\beta, \gamma$ is strictly smaller than 1. 

At each step of the growth algorithm we obtain a new graph
by adding  one edge to an
existing graph. We will enumerate the
obtained graphs by the number of edges they contain. 
We start with an arbitrary initial finite directed
graph, with at least one node and $n_0$ edges,
denoted $G(n_0)$. For $n =n_0+1,n_0+2,\ldots$, 
$G(n)$ will be a graph with $n$ edges and a random number $N(n)$ of
nodes. If $u$ is a node in $G(n-1)$, $D_{\rm in}(u)$ and $D_{\rm out}(u)$ 
denote the in  and out degree of $u$ respectively. The graph $G(n)$ is
obtained from  $G(n-1)$ as follows.
\begin{itemize}
\item 
With probability
$\alpha$  we append to $G(n-1)$ a new node $v$ and an edge leading
from $v$ to an existing node $w$ in $G(n-1)$ (denoted $v \mapsto w$).
The existing node $w$ in $G(n-1)$ is chosen with probability depending
on its in-degree:
$$
p(\text{$w$ is chosen}) = \frac{D_{\rm in}(w)+\delta_{\text
    in}}{n-1+\delta_{\text in}N(n-1)} \,.
$$
\item 
With probability $\beta$ we only append to $G(n-1)$ a directed edge
$v\mapsto w$
between two existing nodes $v$ and $w$ of $G(n-1)$.
 The existing nodes $v,w$ are chosen independently from the nodes of $G(n-1)$ with
 probabilities 
$$
p(\text{$v$ is chosen}) = \frac{D_{\rm out}(v)+\delta_{\text
    out}}{n-1+\delta_{\text out}N(n-1)}, \ \ 
p(\text{$w$ is chosen}) = \frac{D_{\rm in}(w)+\delta_{\text
    in}}{n-1+\delta_{\text in}N(n-1)}\,.
$$
\item With probability
$\gamma$  we append to $G(n-1)$ a new node $w$ and an edge $v\mapsto
w$ leading
from  the existing node $v$ in $G(n-1)$  to the new node $w$. The
existing node $v$ in $G(n-1)$ is chosen with probability
$$
p(\text{$v$ is chosen}) = \frac{D_{\rm out}(v)+\delta_{\text
    out}}{n-1+\delta_{\text out}N(n-1)}\,.
$$
\end{itemize}
If either
$\delta_{\text in}=0$, or $\delta_{\text out}=0$, we must have
$n_0 >1$ for the initial steps of the algorithm to make
sense. 

For $i,j=0,1,2,\ldots$ and $n\geq n_0$, let $N_{ij}(n)$ be the
(random) number of nodes in $G(n)$ with in-degree $i$ and out-degree
$j$. Theorem 3.2 in \cite{bollobas:borgs:chayes:riordan:2003}  shows
that there are nonrandom constants $(f_{ij})$ such that
\begin{equation} \label{e:limit.f}
\lim_{n\to\infty} \frac{N_{ij}(n)}{n}=f_{ij} \ \ \text{a.s. for
  $i,j=0,1,2,\ldots$.}
\end{equation}
Clearly, $f_{00}=0$. Since we obviously have 
$$
\lim_{n\to\infty} \frac{N(n)}{n}=1-\beta \ \ \text{a.s.,}
$$
we see that the empirical joint in- and out-degree distribution in the
sequence $(G(n))$ of growing random graphs has as a nonrandom limit
the probability distribution 
\begin{equation} \label{e:limit.p}
\lim_{n\to\infty}
\frac{N_{ij}(n)}{N(n)}=\frac{f_{ij}}{1-\beta}=:p_{ij} \ \
\text{a.s. for   $i,j=0,1,2,\ldots$.}
\end{equation}
In \cite{bollobas:borgs:chayes:riordan:2003} it was shown that the
limiting degree distribution $(p_{ij})$ has, marginally, regularly
varying (in fact, power-like) tails. Specifically, Theorem 3.1 {\it
  ibid.} shows that for some finite positive constants $C_{\text in}$
and $C_{\text out}$ we have
\begin{equation} \label{e:marginal.regvar}
p_i(\text{in}):= \sum_{j=0}^\infty p_{ij} \sim C_{\text
  in}i^{-\alpha_{\text in}} \ \ \text{as $i\to\infty$, as long as
  $\alpha \delta_{\text in}+\gamma>0$,}  
\end{equation} 
$$
p_j(\text{out}):= \sum_{i=0}^\infty p_{ij} \sim C_{\text
  out}j^{-\alpha_{\text out}} \ \ \text{as $j\to\infty$, as long as
  $\gamma \delta_{\text out}+\alpha>0$.}
$$
Here
\begin{equation} \label{e:exponents}
\alphain = 1+ \frac{1+\deltain(\alpha+\gamma)}{\alpha+\beta}, \ \ 
\alphaout = 1+ \frac{1+\deltaout(\alpha+\gamma)}{\gamma+\beta}\,.
\end{equation}
We will prove that the limiting degree distribution $(p_{ij})$ in
\eqref{e:limit.p} has jointly regularly varying tails and obtain the
corresponding tail measure.

This paper is organized as follows. We start with a summary of
multivariate regular variation in Section \ref{sec:mult.reg.var}. 
In Section \ref{sec:genfuncion} we
show that the joint generating function of  in-degree and 
out-degree satisfies a  partial differential equation. We solve the
differential equation and obtain an expression for the
generating function. 
In Section \ref{sec:InOutRegVar} we represent the distribution corresponding to the
generating function as a mixture of negative
binomial random variables where the mixing distribution is
Pareto. This allows direct computation of the tail  measure of the
non-standard regular variation of in- and out-degree without using
transform methods. The tail
measure  is absolutely continuous
with respect to two dimensional Lebesgue measure, and we exhibit its density. 
We also present in Section \ref{subsec:plots}  graphical evidence of
the variety of dependence structures possible for the tail measure
based on explicit formulae, simulation and iteration of the defining
difference equation for limiting frequencies.

Using the joint generating function of $\{p_{ij}\}$, an alternate
route for studing heavy tail behavior of in- and out-degree is to use 
transform methods and Tauberian theory.
The multivariate  Tauberian theory has been developed and we will
report this
elsewhere.

\section{Multivariate regular variation} \label{sec:mult.reg.var} 
We briefly review the basic concepts of multivariate regular
variation \citep{resnick:2007}  which forms the mathematical framework for multivariate
heavy tails. We restrict attention to two dimensions since this is the
context for the rest of the paper. 

A random vector $(X,Y)\geq \bzero$ has a distribution that is 
non-standard regularly varying  if there exist {\it
  scaling functions\/} $a(h)\uparrow \infty$ and  $b(h)\uparrow
\infty$ and a non-zero limit measure $\nu(\cdot)$ called the {\it limit or tail
measure\/} such that as $h\to\infty$,
\begin{equation}\label{e:defreg}
hP\bigl[\bigl(X/a(h),Y/b(h)\bigr) \in \cdot \,\bigr] \vrightarrow  \nu(\cdot)
\end{equation}
where ``$\vrightarrow $'' denotes vague convergence of measures in
$M_+([0,\infty]^2\setminus \{\bzero\})=M_+(\E)$, the space of Radon
measures on $\E$. The scaling functions will be regularly varying and
we assume their indices are positive and therefore, without loss of
generality, we may suppose $a(h)$ and $b(h)$ are continuous and
strictly increasing. The phrasing in \eqref{e:defreg} implies the
marginal distributions have regularly varying tails.

In case $a(h)=b(h)$, $(X,Y)$ has a distribution with {\it standard\/}
  regularly varying tails \citep[Section 6.5.6]{resnick:2007}. Given a vector with a distribution which is
  non-standard regularly varying, there are at least two methods
  \citep[Section 9.2.3]{resnick:2007}
for
  standardizing the vector so that the transformed vector has standard
  regular variation. The simplest is the power method which is
  justified when the scaling functions are power functions:
$$a(h)=h^{\gamma_1},\quad b(h)=h^{\gamma_2},\quad \gamma_i>0, \,i=1,2.$$
For instance with $c=\gamma_2/\gamma_1$,
\begin{equation}\label{e:standard}
hP\bigl[ \bigl( X^c/h^{\gamma_2}, Y/h^{\gamma_2}\bigr) \in \cdot \,]
  \vrightarrow  \tilde \nu(\cdot),
\end{equation}
where if $T(x,y)=(x^c,y)$, then $\tilde \nu=\nu\circ T^{-1}.$ Since
the two scaling functions in \eqref{e:standard} are the same, the regular
variation is now standard. The measure $\tilde \nu$ will  have a
scaling property and for an appropriate change of coordinate system,
the correspondingly transformed $\tilde \nu$ can be factored into a
product; for example the polar
coordinate transform  is one such coordinate system change which
factors $\tilde \nu$ into  a product of a Pareto measure and  an
angular measure and this  is one way to describe the
asymptotic dependence structure of the standardized $(X,Y)$
\citep[Section 6.1.4]{resnick:2007}.
Another suitable transformation is given in Section \ref{sec:InOutRegVar} 
based on ratios.

\section{The joint generating function of in-degree and 
out-degree} \label{sec:genfuncion} 

Define the joint generating function of the limit distribution
$\{p_{ij}\}$  of
 in-degree and  out-degree in \eqref{e:limit.p} by
\begin{equation} \label{e:phi}
\varphi(x,y) = \sum_{i=0}^\infty \sum_{j=0}^\infty x^i y^jp_{ij}, \
0\leq x,y\leq 1\,.
\end{equation}

The following lemma shows that the generating function satisfies a
 partial differential equation.
\begin{lemma} \label{l:pde}
The function $\varphi$ is continuous on the square $[0,1]^2$ and is
infinitely continuously differentiable in the interior of the
square. In this interior it satisfies the equation
\begin{equation} \label{e:pde}
\bigl[ c_1\deltain(1-x) + c_2\deltaout(1-y)+1\bigr] \varphi
+ c_1x(1-x)\frac{\partial \varphi}{\partial x} 
+ c_2y(1-y)\frac{\partial \varphi}{\partial y} 
\end{equation}
$$
= \frac{\alpha}{\alpha+\gamma} y +\frac{\gamma}{\alpha+\gamma} x\,,  
$$
where
\begin{equation} \label{e:c}
c_1 = \frac{\alpha+\beta}{1+\deltain(\alpha+\gamma)}, \ \ 
c_2 = \frac{\beta+\gamma}{1+\deltaout(\alpha+\gamma)}\,.
\end{equation}
\end{lemma}
\begin{proof}
Only the form of the partial differential equation in
\eqref{e:pde}  requires justification.
The following recursive relation connecting the limiting probabilities
$(p_{ij})$ was established in
\cite{bollobas:borgs:chayes:riordan:2003},
\begin{align} 
p_{ij} =& c_1(i-1+\deltain)p_{i-1,j} - c_1(i+\deltain)p_{ij}
+ c_2(j-1+\deltaout)p_{i,j-1}
\label{e:recursion} \\
&-c_2(j+\deltaout) + \frac{\alpha}{\alpha+\gamma}\one(i=0,j=1)+  
\frac{\gamma}{\alpha+\gamma}\one(i=1,j=0)
\nonumber\end{align}
for $i,j=0,1,2,\ldots$, with the understanding that any $p$ with a
negative subscript is equal to zero.  Rearranging the terms,
multiplying both sides by $x^iy^j$ and summing up, 
we see that for $0<x,y<1$, 
\begin{align}
\sum_{i=0}^\infty \sum_{j=0}^\infty (c_1\deltain +&
c_2\deltaout +1 +c_1i +c_2j) x^i y^j   p_{ij}
= \frac{\alpha}{\alpha+\gamma}y + \frac{\gamma}{\alpha+\gamma}x
 \label{e:sums}\\
&+ c_1 \sum_{i=1}^\infty \sum_{j=0}^\infty (i-1+\deltain) x^i y^j
p_{i-1,j} 
+ c_2 \sum_{i=0}^\infty \sum_{j=1}^\infty (j-1+\deltaout) 
x^i y^j p_{i,j-1} \,.
\nonumber
\end{align}
Since
$$
\frac{\partial \varphi}{\partial x} (x,y) = \sum_{i=1}^\infty
\sum_{j=0}^\infty ix^{i-1}y^j p_{ij}, \ \ 
\frac{\partial \varphi}{\partial y} (x,y) = \sum_{i=0}^\infty
\sum_{j=1}^\infty j x^iy^{j-1}p_{ij}\,,
$$
we can rearrange the terms in \eqref{e:sums} to obtain \eqref{e:pde}. 
\end{proof}

The next theorem  gives an
explicit 
formula for the joint generating function $\varphi$ in \eqref{e:phi}.
\begin{theorem} \label{t:formula.gf}
Let
\begin{equation} \label{e:a}
a=  c_2/c_1\,,
\end{equation}
where $c_1$ and $c_2$ are given in \eqref{e:c}. Then   for $0\leq x,y\leq 1$,
\begin{align} 
\varphi(x,y) =&
 \frac{\alpha}{\alpha+\gamma} c_1^{-1} y \int_1^\infty 
z^{-(1+1/c_1)} \bigl( x+(1-x)z\bigr)^{-\deltain} \bigl(
y+(1-y)z^a\bigr)^{-(\deltaout +1)}\, dz
\label{e:formula.gf}\\
&+ \frac{\gamma}{\alpha+\gamma} c_1^{-1} x \int_1^\infty 
z^{-(1+1/c_1)} \bigl( x+(1-x)z\bigr)^{-(\deltain+1)} \bigl(
y+(1-y)z^a\bigr)^{-\deltaout }\, dz.
\nonumber
\end{align}
\end{theorem}

\begin{proof}
The partial differential equation in \eqref{e:pde} is a linear
equation of the form (2), p.6 in \cite{jones:1971}, and to solve it we
follow the procedure suggested {\it ibid.}. Specifically, we write the 
equation \eqref{e:pde} in the form
\begin{equation} \label{e:pde.lin}
a(x,y) \frac{\partial \varphi}{\partial x} + b(x,y) \frac{\partial
  \varphi}{\partial y}=  c(x,y)\varphi + d(x,y)\,,
\end{equation}
with 
$$
a(x,y) = c_1x(1-x), \ \ b(x,y) = c_2y(1-y)\,,
$$
$$
c(x,y) = c_1\deltain x +  c_2\deltaout y - \rho, \ \
d(x,y) = \alpha (\alpha+\gamma)^{-1} y 
+ \gamma (\alpha+\gamma)^{-1} x\,,
$$
where
$$
\rho = c_1\deltain  +  c_2\deltaout +1\,.
$$
Consider the family of characteristic curves for the differential
equation \eqref{e:pde.lin} defined by the ordinary differential
equation
$$
\frac{dy}{dx} = \frac{b(x,y)}{a(x,y)}\,.
$$
It is elementary to check that the characteristic curves form a 
one-parameter family, $\{ \CurveC_\theta, \ \theta>0\}$, with the
curve $\CurveC_\theta$ given by 
\begin{equation} \label{e:curve}
y=\frac{1}{1+\theta x^{-a}(1-x)^a}, \ 0<x<1\,.
\end{equation} 
Along each characteristic curve $\CurveC_\theta$ the function 
$u(x) = \varphi\bigl( x, y(x)\bigr), \, 0<x<1,$ satisfies the ordinary
differential equation 
\begin{equation} \label{e:ode}
\frac{du}{dx} = \frac{c(x,y)u + d(x,y)}{a(x,y)}
= u\psi_1(x) + \psi_2(x)\,,
\end{equation}
where 
$$
\psi_1(x) = \frac{c_1\deltain x + c_2\deltaout \bigl( 1+ \theta
  x^{-a}(1-x)^a\bigr)^{-1} - \rho}{c_1x(1-x)} \,,
$$
$$
\psi_2(x) = \frac{\gamma x + \alpha \bigl( 1+ \theta
  x^{-a}(1-x)^a\bigr)^{-1}}{(\alpha+\gamma) c_1x(1-x)} \,. 
$$

Let $H$ be a function satisfying 
\begin{equation} \label{e:H}
H^\prime (x) = \psi_1(x), \ 0<x<1\,,
\end{equation}
and define 
$$
A(x) = u(x) e^{-H(x)}, \ 0<x<1\,. 
$$
It follows from \eqref{e:ode} that
\begin{equation} \label{e:A}
A^\prime(x) = \psi_2(x) e^{-H(x)}, \ 0<x<1\,. 
\end{equation}
We compute the function $u$ by solving the differential equations
\eqref{e:H} and \eqref{e:A}.

To solve \eqref{e:H}, write it first in the form 
$$
H^\prime (x) = \frac{\deltain}{1-x} - \frac{\rho/c_1}{x(1-x)} 
+ \frac{c_2\deltaout/c_1}{1+ \theta x^{-a}(1-x)^a}\frac{1}{x(1-x)}\,.
$$
It is elementary to check by differentiation that
$$
\int \frac{1}{1+ \theta x^{-a}(1-x)^a}\frac{1}{x(1-x)}\, dx
= -\log (1-x) + a^{-1} \log \bigl( x^a + \theta (1-x)^a\bigr) +C_1
$$
with $C_1\in\bbr$. Therefore, for $0<x<1$, 
\begin{equation} \label{e:formula.H}
H(x) = c_1^{-1} \log(1-x) - \rho c_1^{-1} \log x + \deltaout 
\log \bigl( x^a + \theta (1-x)^a\bigr) +C_1\,, 
\end{equation}
implying that 
\begin{align*}
A^\prime(x) =& e^{-C_1} \frac{\gamma x + \alpha \bigl( 1+ \theta
  x^{-a}(1-x)^a\bigr)^{-1}}{(\alpha+\gamma)c_1 x(1-x)} (1-x)^{-1/c_1} x^{\rho/c_1} 
\bigl( x^a + \theta (1-x)^a\bigr)^{-\deltaout}\\
=& \frac{e^{-C_1}}{(\alpha+\gamma)c_1}\gamma (1-x)^{-(1+1/c_1)} x^{\deltain +1/c_1}
\bigl( 1+ \theta x^{-a}  (1-x)^a\bigr)^{-\deltaout}\\
\quad &
+ \frac{e^{-C_1}}{(\alpha+\gamma)c_1}  (1-x)^{-(1+1/c_1)}  x^{\deltain -1 +1/c_1}
\bigl( 1+ \theta x^{-a}  (1-x)^a\bigr)^{-(1+\deltaout)}\,.
\end{align*}
We can now write
\begin{align}
A(x) =& 
e^{-C_1} \frac{\gamma}{(\alpha+\gamma)c_1} \int_0^x
(1-t)^{-(1+1/c_1)} t^{\deltain +1/c_1} 
\bigl( 1+ \theta t^{-a}  (1-t)^a\bigr)^{-\deltaout}\, dt
 \label{e:formula.A} \\
&
+ e^{-C_1} \frac{\alpha}{(\alpha+\gamma)c_1} \int_0^ x
(1-t)^{-(1+1/c_1)}  t^{\deltain -1 +1/c_1} 
\bigl( 1+ \theta t^{-a}  (1-t)^a\bigr)^{-(1+\deltaout)}\, dt +C_2
\nonumber \end{align}
with $C_2\in \bbr$. Using \eqref{e:formula.H} and \eqref{e:formula.A}
we obtain the following expression for the the function 
$u(x) = \varphi\bigl( x, y(x)\bigr), \, 0<x<1$ along the
characteristic curve $\CurveC_\theta$. 
\begin{align*}
u(x) =& A(x) e^{H(x)} = \frac{\gamma}{\alpha+\gamma}c_1^{-1}
(1-x)^{1/c_1} x^{-\rho/c_1} \bigl( 1+ 
\theta x^{-a}  (1-x)^a\bigr)^{\deltaout} \\
&{}\cdot  \int_0^x (1-t)^{-(1+1/c_1)} t^{\deltain +1/c_1}
\bigl( 1+ \theta t^{-a}  (1-t)^a\bigr)^{-\deltaout}\, dt
\\
&
\quad + \frac{\alpha}{\alpha+\gamma}c_1^{-1} (1-x)^{1/c_1} x^{-\rho/c_1}
\bigl( 1+ \theta x^{-a}  (1-x)^a\bigr)^{\deltaout} {}
\\
&
\cdot \int_0^ x (1-t)^{-(1+1/c_1)}  t^{\deltain -1 +1/c_1}
\bigl( 1+ \theta t^{-a}  (1-t)^a\bigr)^{-(1+\deltaout)}\, dt
\\
&\quad+ C_3 (1-x)^{1/c_1} x^{-\rho/c_1} \bigl( 1+
\theta x^{-a}  (1-x)^a\bigr)^{\deltaout} 
\end{align*}
with $C_3=C_3(\theta)\in \bbr$. Multiply both sides of this
equation by $x^{a\deltaout + \rho/c_1}$ and let $x\to 0$. Using the
fact that the generating function is bounded, we see that $C_3=0$. 
We can now obtain an expression for the joint generating function
$\varphi$ everywhere in $(0,1)^2$ by noticing that a point $(x,y)$,
 $0<x,y<1,$ lies on the characteristic curve $\CurveC_\theta$
with
$$
\theta = \frac{(1-y)/y}{\bigl( (1-x)/x\bigr)^a}\,.
$$
We conclude that
\begin{align*}
\varphi(x,y) =& \frac{\gamma}{\alpha+\gamma}c_1^{-1}  (1-x)^{1/c_1}
x^{-\rho/c_1} \bigl( 1+ 
\theta x^{-a}  (1-x)^a\bigr)^{\deltaout} \\
&\cdot 
 \int_0^x (1-t)^{-(1+1/c_1)} t^{\deltain +1/c_1}
\left( 1+ \frac{(1-y)/y}{\bigl( (1-x)/x\bigr)^a} 
t^{-a}  (1-t)^a\right)^{-\deltaout}\, dt\\
&\quad 
+  \frac{\alpha}{\alpha+\gamma}c_1^{-1} (1-x)^{1/c_1} x^{-\rho/c_1}
\bigl( 1+ \theta x^{-a}  (1-x)^a\bigr)^{\deltaout} 
\\
&\cdot
\int_0^ x (1-t)^{-(1+1/c_1)}  t^{\deltain -1 +1/c_1}
\left( 1+ \frac{(1-y)/y}{\bigl( (1-x)/x\bigr)^a} t^{-a}
(1-t)^a\right)^{-(1+\deltaout)}\, dt \,.
\end{align*}
Changing the variable in both integrals to
$$
z=\frac{x(1-t)}{t(1-x)}
$$ 
and rearranging the terms, we obtain \eqref{e:formula.gf} for
$0<x,y<1$. Now we can extend this formula for the joint
generating function to the boundary of the square $[0,1]^2$ by
continuity. 
\end{proof}

\section{Joint regular variation of the distribution of in-degree
and out-degree} \label{sec:InOutRegVar} 

In this section we analyze the explicit form \eqref{e:formula.gf} of the joint generating
function of the limiting distribution of in-degree and 
out-degree obtained in Theorem \ref{t:formula.gf} to prove the
nonstandard joint regular variation of  in-degree and 
out-degree. We also obtain an expression for the density of the tail
measure.

We start by writing the joint generating function in
\eqref{e:formula.gf} as
\begin{equation} \label{e:split.gf}
\varphi(x,y) = \frac{\gamma}{\alpha+\gamma}  x \varphi_1(x,y) 
+ \frac{\alpha}{\alpha+\gamma}  y \varphi_2(x,y) \,,
\end{equation}
with
\begin{align} 
\varphi_1(x,y) =& c_1^{-1}\int_1^\infty 
z^{-(1+1/c_1)} \bigl( x+(1-x)z\bigr)^{-(\deltain+1)} \bigl(
y+(1-y)z^a\bigr)^{-\deltaout }\, dz\,, \label{e:phi1}\\
\varphi_2(x,y) =& c_1^{-1}\int_1^\infty 
z^{-(1+1/c_1)} \bigl( x+(1-x)z\bigr)^{-\deltain} \bigl(
y+(1-y)z^a\bigr)^{-(\deltaout +1)}\, dz \label{e:phi2}
\end{align}
for $0\leq x,y\leq 1$.  Each of these functions $\varphi_i$
is a mixture of a product of negative binomial generating functions of
possibly fractional order. 
On some probability space we can find nonnegative integer-valued
random variables $X_j,\, Y_j, \ j=1,2$ such that
$$
\varphi_j(x,y) = E\bigl( x^{X_j}y^{Y_j}\bigr), \ 0\leq x,y\leq 1, \
j=1,2\,.
$$
If $(I,O)$ is a random vector with generating function given in
\eqref{e:split.gf}, then we can represent in distribution $(I,O)$ as
\begin{equation}\label{eq:io.rep}
(I,0)\stackrel{d}{=} B(1+X_1,Y_1)+(1-B) (X_2,1+Y_2),
\end{equation}
where $B$ is a Bernoulli switching variable independent of $X_j,Y_j,\,
j=1,2$ with
$$P[B=1]=1-P[B=0]=\frac{\gamma}{\alpha+\gamma}.$$

Theorem \ref{t:two.measures} below shows that each of the random
vectors $\bigl( X_j,\, Y_j\bigr)$, $j=1,2$, has a bivariate regularly varying distribution. The
decomposition \eqref{e:split.gf} then gives the
joint regular variation of in-degree and  out-degree. 

\begin{theorem} \label{t:two.measures} 
Let $\alphain$ and $\alphaout$ be given by
\eqref{e:exponents}. Then for each $j=1,2$ there is a Radon
measure $V_j$ on $[0,\infty]^2\setminus \{ {\mathbf 0}\}$ such that 
\begin{equation} \label{e:regvar.j}
hP\Bigl( \bigl( h^{-1/(\alphain-1)}X_j, \,
h^{-1/(\alphaout-1)}Y_j\bigr)\in\cdot\Bigr) \vrightarrow V_j (\cdot),
\end{equation}
as $h\to\infty$ vaguely in $[0,\infty]^2\setminus \{ {\mathbf 0}\}$. 
Furthermore, $V_1$  and $V_2$ concentrate on $(0,\infty)^2$ where they
have Lebesgue densities  given, respectively,  by 
\begin{align} 
f_1(x,y) =& c_1^{-1} \bigl(
\Gamma(\deltain+1)\Gamma(\deltaout)\bigr)^{-1}
x^{\deltain}y^{\deltaout-1}
\int_0^\infty z^{-(2+1/c_1+\deltain +a\deltaout)} e^{-(x/z+y/z^a)}\,
dz \label{e:density}
\intertext{and}
f_2(x,y) =& c_1^{-1} \bigl(
\Gamma(\deltain)\Gamma(\deltaout+1)\bigr)^{-1}
x^{\deltain-1}y^{\deltaout}
\int_0^\infty z^{-(1+a+1/c_1+\deltain +a\deltaout)} e^{-(x/z+y/z^a)}\,
dz\,. \label{e:density.2}
\end{align}
Therefore, a random vector $\bigl( I,O)$ with the joint
probabilities given by $(p_{ij})$ in \eqref{e:limit.p} satisfies 
\begin{equation}\label{e:ganzMegilla}
hP\Bigl( \bigl( h^{-1/(\alphain-1)}I, \,
h^{-1/(\alphaout-1)}O\bigr)\in\cdot\Bigr) \vrightarrow 
\frac{\gamma}{\alpha+\gamma} V_1 (\cdot) +
\frac{\alpha}{\alpha+\gamma} V_2(\cdot) 
\end{equation}
as $h\to\infty$ vaguely in $[0,\infty]^2\setminus \{ {\mathbf 0}\}$. 
\end{theorem}
\begin{proof}
It is enough to prove \eqref{e:regvar.j} and
\eqref{e:density}. We treat the case $j=1$. The case $j=2$ is
 analogous and is omitted.
 
Let $T_\delta (p)$ be a negative binomial integer valued random
variable with parameters $\delta >0$ and $p\in (0,1)$. We abbreviate
this as $NB(\delta,p)$. The generating function of $T_\delta (p)$ is
$$Es^{T_\delta (p)}=  (s+(1-s)p^{-1})^{-\delta}.$$
It is well known and elementary to prove by switching to Laplace
transforms
 that as $p\downarrow 0$,
$$pT_\delta(p) \Rightarrow \Gamma_\delta$$ 
where $\Gamma_\delta $ is a Gamma random variable with distribution
$F_\delta (x)$ and  density
$$F'_\delta(x)= \frac{e^{-x}x^{\delta -1}}{\Gamma (\delta)},\quad
x>0.$$

Now suppose $\{T_{\delta_1}(p),\,p \in (0,1)\} $ and $\{\tilde
T_{\delta_2}(p),\,p \in (0,1)\} $ are two independent families of $NB$
random variables. 
We can represent the mixture in \eqref{e:phi1} as 
$$(X_1,Y_1)=\bigl(T_{\deltain +1} (Z^{-1}), \tilde T_{\deltaout }
(Z^{-a}) \bigr),$$
where $Z$ is a Pareto random variable on $[1,\infty)$ with index
$c_1^{-1}$, independent of the $NB$ random variables. 
To ease writing, we set $\delta_1=\deltain +1$ and
$\delta_2=\deltaout$.

Define the measure $\nu_{c} $ on $(0,\infty]$  by
$\nu_c(x,\infty]=x^{-c}, \,x>0$. We now claim, as $h\to\infty$, in
$M_+( (0,\infty]\times [0,\infty]^2)$,
\begin{equation}\label{eq:17}
hP\Bigl[\Bigl( \frac{Z}{h^{c_1}}, \bigl( Z^{-1}T_{\delta_1} (Z^{-1}),
Z^{-a}  \tilde T_{\delta_2} (Z^{-a})
\bigr)
\Bigr) \in \cdot \,\Bigr]
\stackrel{v}{\to}
\nu_{c_1^{-1}} \times P[\Gamma_{\delta_1} \in
\cdot \,]\times P[\Gamma_{\delta_2 } \in \cdot \,].
\end{equation}
To prove this, suppose $x>0$ and let $g(u,v)$ be a function bounded
and continuous on $[0,\infty]^2 $
 and it suffices to show,
\begin{equation}\label{eq:18}
hE\Bigl(1_{[Z/h^{c_1} >x] }g\bigl( Z^{-1}T_{\delta_1} (Z^{-1}), Z^{-a}T_{\delta_2}
(Z^{-a}) \bigr)\Bigr) \to x^{-c_1^{-1}} E\bigl(g(\Gamma_{\delta_1} ,\tilde
\Gamma_{\delta_2})\bigr)\end{equation}
where $\Gamma_{\delta_1}\independent \tilde \Gamma_{\delta_2}$.

Observe  as $p\downarrow 0$,
$$E\Bigl(g\bigl( pT_{\delta_1} (p), p^a \tilde T_{\delta_2}
(p^{a}) \bigr)\Bigr) \to   E\bigl(g(\Gamma_{\delta_1} ,\tilde
\Gamma_{\delta_2})\bigr)
$$
and so, given $\epsilon>0$, there exists $\eta>0$ such that 
\begin{equation}\label{eq:smalldiff}
\sup_{p<\eta} \bigl|E\Bigl(g\bigl( pT_{\delta_1} (p), p^a \tilde T_{\delta_2}
(p^{a}) \bigr)\Bigr) -  E\bigl(g(\Gamma_{\delta_1} ,\tilde
\Gamma_{\delta_2})\bigr) \bigr| <\epsilon.
\end{equation}
Bound the difference between the LHS and RHS of \eqref{eq:18} by
\begin{align*}
\Bigl|hE\Bigl(1_{[Z/h^{c_1} >x] }&g\bigl( Z^{-1}T_{\delta_1} (Z^{-1}),
Z^{-a}  \tilde T_{\delta_2}
(Z^{-a}) \bigr)\Bigr) 
-
hE(1_{[Z/h^{c_1} >x] }
  E\bigl(g(\Gamma_{\delta_1} ,\tilde
\Gamma_{\delta_2})\bigr) \Bigr|\\
&+\Bigl| h E(1_{[Z/h^{c_1} >x] }
  E\bigl(g(\Gamma_{\delta_1} ,\tilde
\Gamma_{\delta_2})\bigr)  -
x^{-c_1^{-1}} E\bigl(g(\Gamma_{\delta_1} ,\tilde
\Gamma_{\delta_2})\bigr) \Bigr|=A+B,
\end{align*} where $B=o(1)$ and is henceforth neglected.
Write $E^Z(\cdot)=E(\cdot|Z)$ for the conditional expectation and
bound $A$ by
\begin{equation}\label{e:forgoing}
E\Bigl(h1_{[Z/h^{c_1} >x] } \bigl| E^{Z}
g\bigl( Z^{-1}T_{\delta_1} (Z^{-1}), Z^{-a} \tilde T_{\delta_2}
(Z^{-a}) \bigr) -  E\bigl(g(\Gamma_{\delta_1} ,\tilde
\Gamma_{\delta_2})\bigr) 
\bigr| \Bigr).
\end{equation}
As soon as $h$ is large enough so that $h^{-{c_1}} x^{-1}<\eta$,
\eqref{e:forgoing}
 bounded
  by
$$E\bigl(h1_{[Z/h^{c_1} >x] } \bigr) \epsilon \to \epsilon
x^{-c_1^{-1}}. $$ Let $\epsilon \to 0$ and we have verified
\eqref{eq:18}   and therefore \eqref{eq:17}.

The next step is to apply a mapping to the convergence in
\eqref{eq:17}. Define $\chi :(0,\infty]\times [0,\infty]^2\mapsto
(0,\infty]\times [0,\infty]^2$ by
$$\chi \bigl(x,(y_1,y_2)\bigr)=\bigl(x,(xy_1,x^ay_2)\bigr).$$
This transformation satisfies the compactness condition in
\cite[Proposition 5.5, page 
  141]{resnick:2007} or the bounded away condition  in
\cite[Section 2.2]{lindskog:resnick:roy:2013}. Following the product
discussion of Example 3.3 in \cite{lindskog:resnick:roy:2013} or
\cite[Corollary 2.1, page 682]{maulik:resnick:rootzen:2002}, we apply
 $\chi $ to the convergence in \eqref{eq:17} which yields
 in $M_+(
(0,\infty]\times [0,\infty]^2)$, as $h\to\infty$,
\begin{equation}\label{eq:almost}
hP\Bigl[\Bigl( \frac{Z}{h^{c_1}}, \bigl( \frac{T_{\delta_1} (Z^{-1})}{h^{c_1}},
\frac{ \tilde T_{\delta_2} (Z^{-a})}{h^{c_2}}
\bigr)
\Bigr) \in \cdot \,\Bigr]
\stackrel{v}{\to}
\Bigl(\nu_{c_1^{-1}} \times P[\Gamma_{\delta_1} \in
\cdot \,]\times P[\Gamma_{\delta_2 } \in \cdot \,]\Bigr)\circ \chi ^{-1} (\cdot),
\end{equation}
where we used the fact that $ac_1=c_2$.

We must extract from \eqref{eq:almost} the desired convergence in
$M_+([0,\infty]^2\setminus \{\bzero\})$,
\begin{equation}\label{eq:desired}
hP\Bigl[\Bigl( \frac{T_{\delta_1} (Z^{-1})}{h^{c_1}},
\frac{ \tilde T_{\delta_2} (Z^{-a})}{h^{c_2}}
\Bigr) \in \cdot \,\Bigr]
\stackrel{v}{\to}
\Bigl(\nu_{c_1^{-1}} \times P[\Gamma_{\delta_1} \in 
\cdot \,]\times P[\Gamma_{\delta_2 } \in \cdot \,]\Bigr)\circ \chi ^{-1}
((0,\infty]\times (\cdot)).
\end{equation}
Assuming  \eqref{eq:desired}, we evaluate
the convergence in \eqref{eq:desired} on a set of the form
$(x,\infty]\times (y,\infty]$ for $x>0,\,y>0$ to get
\begin{align*}
hP\Bigl[ \frac{T_{\delta_1} (Z^{-1})}{h^{c_1}}>x,
\frac{ \tilde T_{\delta_2} (Z^{-a})}{h^{c_2}}>y
\Bigr]
\to &
\iiint_{(u,v,w): uv>x,u^aw>y} \nu_{c_1^{-1}} (du) F_{\delta_1}
(dv)F_{\delta_2} (dw)\\
=&\int_0^\infty \bar F_{\delta_1} ( x/u)  \bar F_{\delta_1}
({y}/{u^a}) \nu_{c_1^{-1}} (du) .
\end{align*}
The right side is the limit measure of the distribution of  $(X_1,Y_1) $ evaluated on
$(x,\infty]\times (y,\infty]$ for $x>0,\,y>0$. Differentiating first with
respect to $x$ and then with respect to $y$ yields after some algebra
the limit measure's density $f_1(x,y)$ in \eqref{e:density}. 

To prove that \eqref{eq:desired} can be obtained from
\eqref{eq:almost}, we need the following result about negative
binomial random variables whose proof is deferred. Suppose $T_\delta
(p) $ is $NB(\delta,p).$ For any $\delta>0$, $k=1,2,\dots $ there is
$c(\delta,k) \in (0,\infty)$ such that
\begin{equation}\label{eq:nbbound}
E\bigl(T_\delta(p)\bigr)^k \leq c(\delta,k) p^{-k} \ \ \text{for all $0<p<1$.}
\end{equation}

Suppose $g:[0,\infty]^2 \setminus \{\bzero\} \mapsto [0,\infty)$ is
continuous, bounded by $\|g\|$ with compact support in
$([0,\epsilon]\times [0,\epsilon])^c$ for some $\epsilon>0.$ Using a
Slutsky style argument, \eqref{eq:almost} implies \eqref{eq:desired}
if
\begin{align*}
0=&\lim_{x\to 0}\limsup_{h\to\infty} 
\Bigl|
hE1_{[Z/h^{c_1} \geq x]} g\bigl(T_{\delta_1}(Z^{-1})/h^{c_1}, \tilde
T_{\delta_2}(Z^{-a})/h^{c_2} \bigr)
-
hEg\bigl(T_{\delta_1}(Z^{-1})/h^{c_1}, \tilde
T_{\delta_2}(Z^{-a})/h^{c_2} \bigr)
\Bigr|\\
=&\lim_{x\to 0}\limsup_{h\to\infty} 
hE1_{[Z/h^{c_1} \leq x]} g\bigl(T_{\delta_1}(Z^{-1})/h^{c_1}, \tilde
T_{\delta_2}(Z^{-a})/h^{c_2} \bigr).
\end{align*}
Keeping in mind the support of $g$, the previous
  expectation is bounded by
$$\|g\| hP\bigl[Z \leq h^{c_1}x, [T_{\delta_1}(Z^{-1})/h^{c_1} >\epsilon] \cup
[T_{\delta_2}(Z^{-a})/h^{c_2} >\epsilon]  \bigr].
$$
Bounding the probability of the union by the sum of two probabilities, we
show how to deal with the first since the second is analogous. Then
neglecting the factor $\|g\|$ we have
\begin{align*}
hP\bigl[Z \leq h^{c_1}x,& T_{\delta_1}(Z^{-1})/h^{c_1} >\epsilon]
=hE\Bigl( 1_{[Z\leq h^{c_1}x]} P\Bigl [T_{\delta_1}(Z^{-1})/h^{c_1}
>\epsilon \Big|Z\Bigr]\Bigr)\\
\intertext{and  picking $k>c_1^{-1}$ and using \eqref{eq:nbbound} we
  get the bound}
\leq & hE\bigl( 1_{[Z\leq h^{c_1}x]} c(\delta_1,k) (Z/h^{c_1} )^k
\epsilon^{-k}\\
=& c(\delta_1,k)\epsilon^{-k}\int_0^x u^kh P[ Z/h^{c_1} \in du]\\
\intertext{and by Karamata's theorem or direct calculation, as
  $h\to\infty$ we get the limit}
=&c(\delta_1,k)\epsilon^{-k}\frac{
c_1^{-1}
}
{k-c_1^{-1}} x^{k-{c_1}^{-1}} 
\end{align*}
which converges to $0$ as $x\to 0$ as desired. 

Finally we verify \eqref{eq:nbbound}. Begin with $\delta =1$ so
$T_1(p)$ is geometric with success probability $p$. It is enough to
prove that for some constant $C(k) \in (0,\infty)$,
\begin{equation}\label{eq:89}
E\Bigl(
\prod_{j=0}^{k-1}(T_1(p) -j) \Bigr) \leq C(k)p^{-k}.
\end{equation}
Differentiating the generating function, we obtain,
\begin{equation}\label{eq:99}
E\Bigl(
\prod_{j=0}^{k-1}(T_1(p) -j) \Bigr) =k! (1-p)^k p^{-k} \leq k! p^{-k}.
\end{equation}

Next, for integer $\delta =1,2,\dots $,  and independent copies
$\tilde T_{1,1}(p),\tilde T_{1,2}(p),\dots,$ of
$T_1(p)$ random variables, we have
\begin{align*}
E\bigl( T_\delta(p)\bigr)^k =& E\bigl(\tilde T_{1,1}(p)+\tilde T_{1,2}(p)
+\ldots + \tilde T_{1,\delta}(p)\bigr)^k\\
\intertext{and applying the $c_r$ inequality in \cite[p. 177]
{loeve:1977} gives}
\leq &
\delta^{k-1}E\bigl(T_1(p)^k\bigr) \leq \delta^{k-1} C(k)p^{-k}.
\end{align*}

Finally, for any $\delta>0$,
\begin{align*}
E\bigl( T_\delta (p)\bigr)^k \leq  E\bigl( T_{\lceil\delta\rceil}
(p)\bigr)^k 
\leq \lceil \delta \rceil ^{k-1} C(k) p^{-k},
\end{align*}
proving \eqref{eq:nbbound} and completing the proof.
\end{proof}

\begin{remark}
{\rm
A change of variables in the integrals in \eqref{e:density} and
\eqref{e:density.2} shows that the random vector $(I,O)$ is bivariate
regular varying with marginal exponents $\alphain-1$ and $\alphaout-1$
accordingly, and with tail measure having density of the form
$$
f(x,y) = c_1^{-1} 
\frac{\gamma/(\alpha+\gamma)}{\Gamma(\deltain+1)\Gamma(\deltaout)} 
x^{\deltain}y^{\deltaout-1}\int_0^\infty t^{1/c_1+\deltain
  +a\deltaout} e^{-(xt+yt^a)}\, dt 
$$
\begin{equation} \label{e:summary.density}
 + c_1^{-1} 
\frac{\alpha/(\alpha+\gamma)}{\Gamma(\deltain)\Gamma(\deltaout+1)}
x^{\deltain-1}y^{\deltaout} \int_0^\infty t^{a-1+1/c_1+\deltain 
  +a\deltaout} e^{-(xt+yt^a)}\, dt 
\end{equation}
for $0<x,y<1$. 
}
\end{remark}

The powers of $h$ used in the scaling functions in \eqref{e:regvar.j}
are, in general, not equal and thus the regular variation in
\eqref{e:ganzMegilla} is non-standard.  However, as the
  scaling functions are pure powers, the vector 
 $(I^a, O)$ is standard regularly varying. One can
then  transform to the familiar polar coordinates. We
consider the alternative transformation $(I^a,O) \mapsto (O/I^a, I)$
which gives the immediate conclusion by Theorem \ref{t:two.measures} that
out-degree is roughly proportional to a  power of the in-degree
when either degree is large. We calculate the
limiting density of  ratio $R:=O/I^a $ given $I$ is large.
\begin{corollary} \label{c:ratio}
As $m\to\infty$,
the conditional distribution of the ratio $O/I^a$ given that $I>m$
converges to a distribution $F_R$ on $(0,\infty)$  
with density
\begin{equation} \label{e:density.ratio}
f_R(r) = \theta_1 r^{\deltaout-1}I_1(r) + \theta_2 r^{\deltaout} I_2(r),
\ r>0\,,
\end{equation}
where
$$
I_1(r) = \int_0^\infty t^{1/c_1+\deltain +a\deltaout} e^{-(t+rt^a)}\,
dt\,,\quad
I_2(r) = \int_0^\infty  t^{a-1+1/c_1+\deltain +a\deltaout} e^{-(t+rt^a)}\,
dt\,,
$$
and
$$
\theta_1=\frac{\gamma}{\Gamma(\deltain+1)\Gamma(\deltaout)D}\,,
\quad
\theta_2=\frac{\alpha}{\Gamma(\deltain)\Gamma(\deltaout+1)D}\,,
$$
with
$$
D= \gamma\frac{\Gamma(1/c_1+\deltain+1)}{\Gamma(\deltain+1)}
+ \alpha \frac{\Gamma(1/c_1+\deltain)}{\Gamma(\deltain)}\,.
$$
\end{corollary}
\begin{proof}
Let $h_m=m^{\alphain-1}$. Notice that for every $\lambda>0$,
\begin{align*}
P\Bigl( O/I^a\leq \lambda \Big| I>m\Bigr)
&= \frac{ h_m P\Bigl( h_m^{-1/(\alphain-1)}I>1, \,
  h_m^{-1/(\alphaout-1)}O/\bigl( h_m^{-1/(\alphain-1)}I\bigr)^a\leq
  \lambda\Bigr)}{h_m P\bigl( h_m^{-1/(\alphain-1)}I>1\bigr)} \\
&\to \frac{(\gamma V_1+ \alpha V_2)\bigl( \bigl\{ (x,y):\, x>1, \,
  y/x^a\leq \lambda\bigr\}\bigr)}{(\gamma V_1+ \alpha V_2)\bigl(
    \bigl\{ (x,y):\, x>1 \bigr\}\bigr)}
\end{align*}
as $m\to\infty$ by Theorem \ref{t:two.measures}. 
The numerator of this ratio can be  rewritten as
$$
\int\int_{x>1, \,  y/x^a\leq \lambda} f(x,y)\, dxdy\,,
$$
and the same can be done to the denominator in this ratio. Using the
density $f$ in \eqref{e:summary.density} and performing an
elementary change of variable shows that the ratio can be written in
the form
$$
\int_0^\lambda f_R(r)\, dr\,,
$$
with $f_R$ as in \eqref{e:density.ratio}. This completes the proof. 
\end{proof}
 
\subsection{Plots, simulation, iteration.}\label{subsec:plots} For
fixed
values of $(\alphain,\alphaout)$, we investigate how
 the dependence structure of $(I,O)$ in \eqref{eq:io.rep} depends on
 the remaining parameters.
We generate plots of $f_R (r)$ and the spectral density for various values of the input parameters using the
explicit formulae and 
compare such plots to histograms obtained by network simulation and
iteration of \eqref{e:recursion}.

\subsubsection{The distribution of $R$.}\label{subsubsec:R}
We fix two values of $(\alphain, \alphaout)$, namely $(7,5)$ and $(5,7)$, and
then plot $f_R(r)$ for several values of the remaining
parameters to see the variety of possible shapes.
 Since $\alpha+\beta+\gamma=1$, fixing values for
$(\alpha,\gamma)$ also determines $\beta$ and because of \eqref{e:exponents},
assuming values for $\alphain, \alphaout, \alpha,\gamma$ determine
values for  $\deltain, \deltaout.$
The density plots are in Figure \ref{fig:amyplots}.
\begin{figure}[t]
\includegraphics[width=3in]{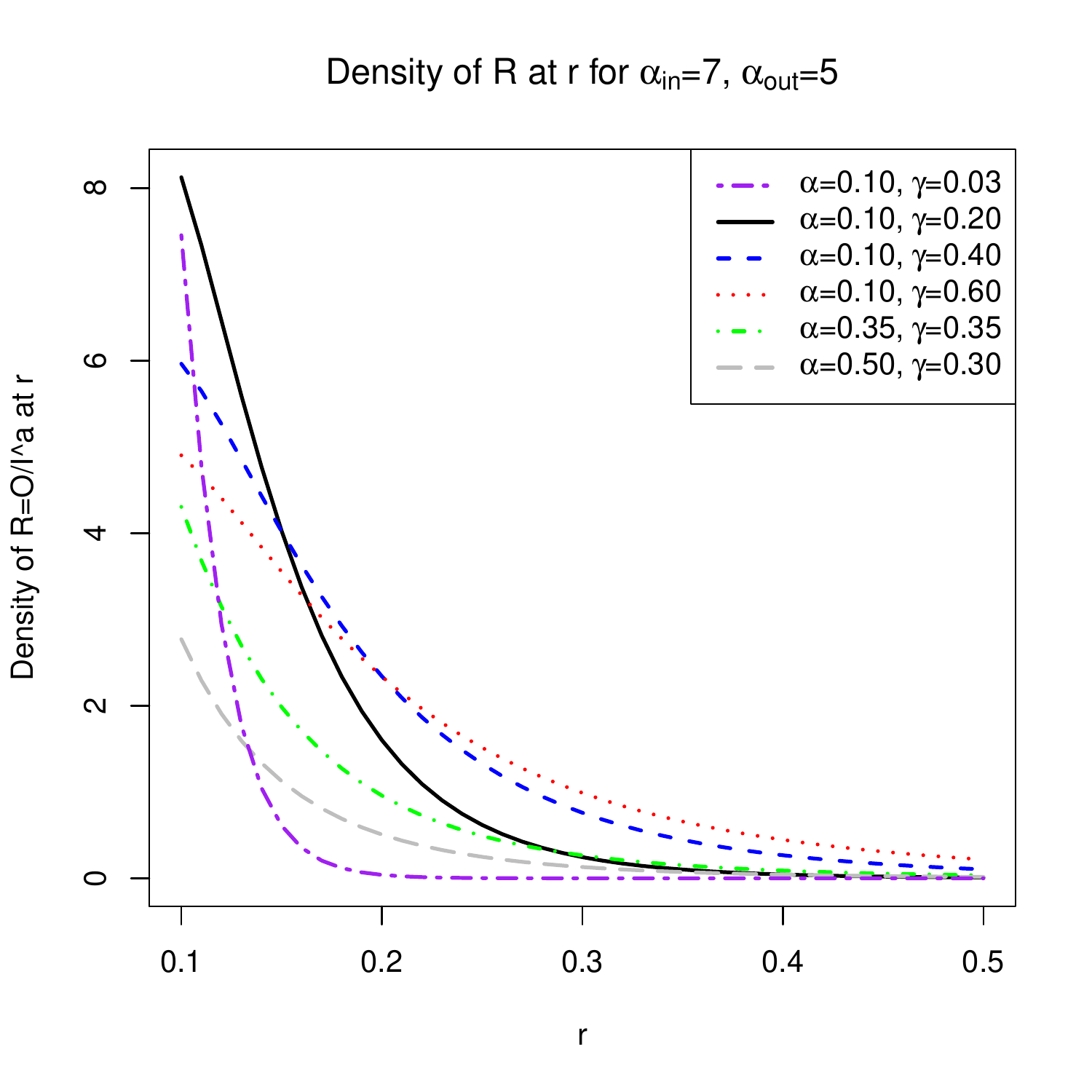}
\includegraphics[width=3in]{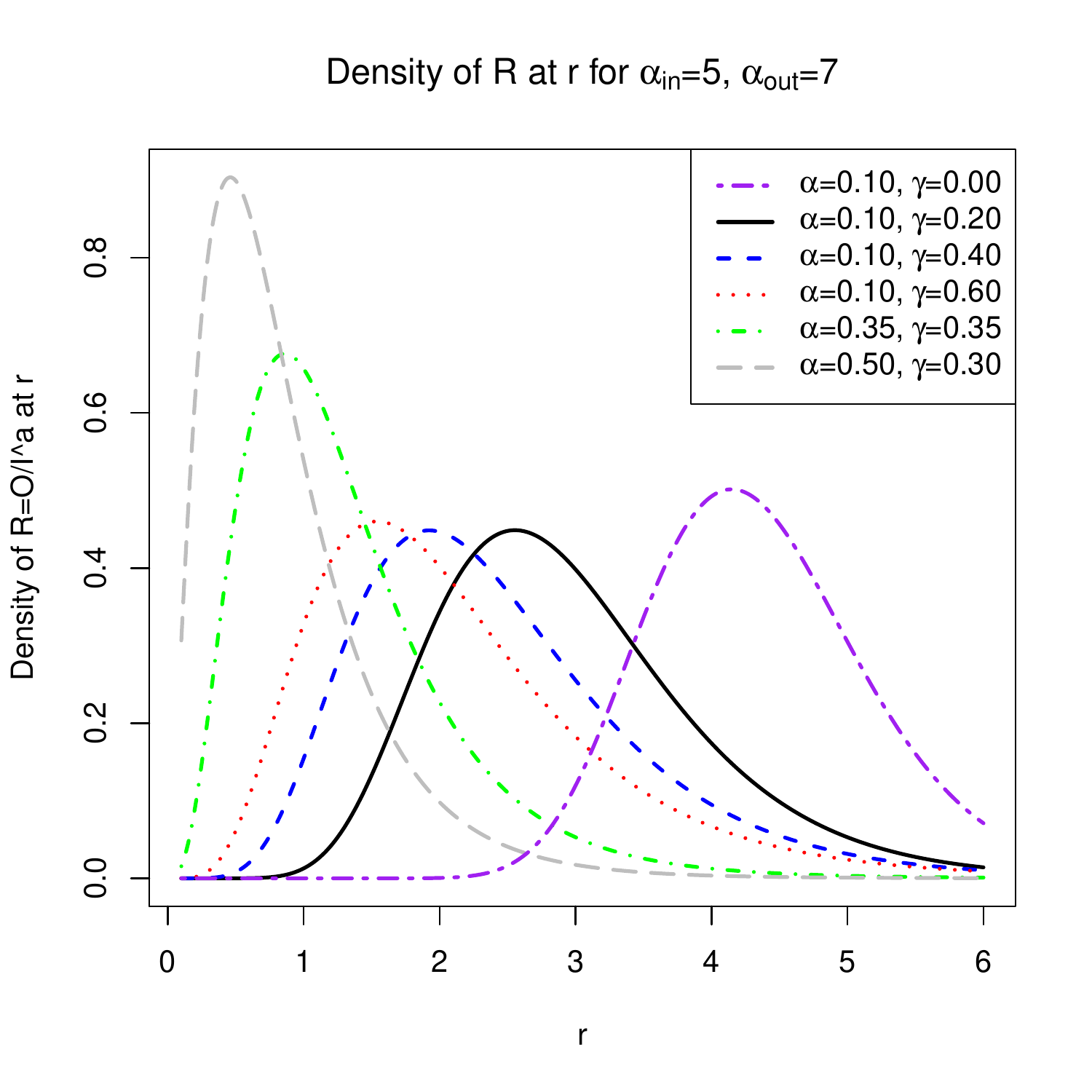}
\caption{The density $f_R(r)$ for $(\alphain,\alphaout)=(7,5)$ (left)
  and $(\alphain,\alphaout)=(5,7)$ (right) for various values of
  $\alpha, \gamma$.}
\label{fig:amyplots}
\end{figure}

Additionally, we
employ two numerical strategies based on 
the convergence of the conditional
distribution of $O/I^a$ given $I>m$ as $m\to \infty$.
Strategy 1 simulates a network
of $10^6$ nodes using software provided by James Atwood (University of
Massachusetts, Amherst) and then computes the histogram of
$O/I^a$ for nodes whose in-degree $I$ exceeds
some large threshold $m$. For the network simulation illustration, we chose
$m$ to be the $99.95\%$ quantile of the in-degrees. 
Strategy 2
computes $p_{ij}$
on a grid $(i,j)$ 
using the  recursion given in \eqref{e:recursion} and then estimates
the density of $O/I^a$ using only the 
grid points with $i$ larger than $m$, the $m$ chosen to be the same
value as used for the network simulation.

We observe
from Figure \ref{fig:amyplots} that the mode of $f_R(r)$ can drift
away from the origin depending on parameter values. So we 
transform $R$ using the 
$\arctan$ function which gives all plots the same compact support
$[0,\pi/2]$,  instead of  an
infinite domain as in Figure \ref{fig:amyplots}.
We compare  the density  of $R$ with the histogram based on network
simulation and the density approximation provided by iteration
 across varying sets of parameter values.   The density of $\arctan R$
 with the plots from the alternative strategies based on simulation
 and iteration are
displayed in Figure 2 for various choices of $(\deltain,\deltaout)$,
 with $\alpha=\beta=0.5$
and $\gamma=0$. For these parameter choices,
the plots of the 
theoretical density with those resulting from network simulation and
probability iteration are in good agreement.
\begin{figure}[t]
\centering
\includegraphics[width=5in, height=5in]{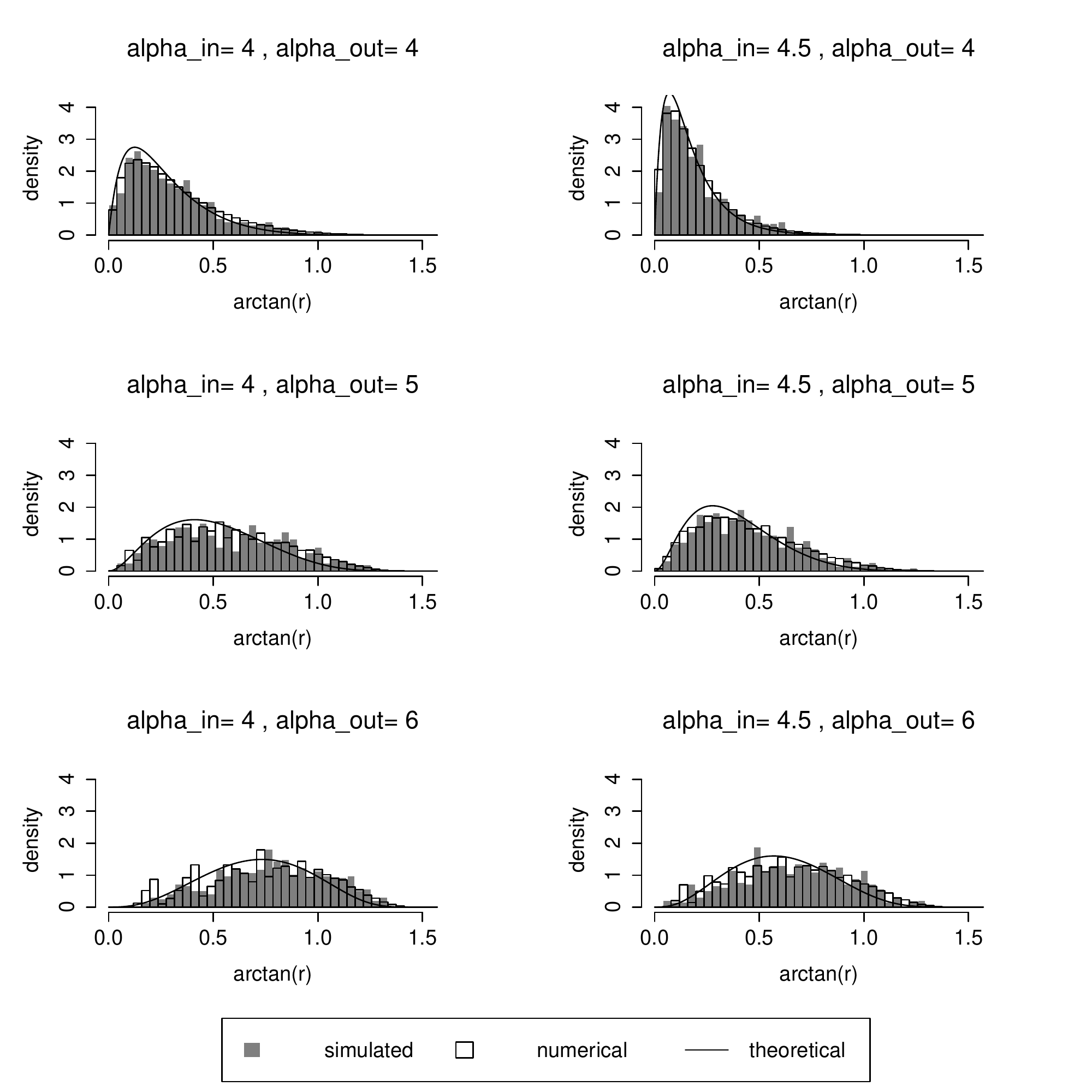}
\label{fig:PhyllisplotsR}
\caption{Comparison of the true density with the estimated densities
  of $\arctan R$ over various values of
  $(\alphain,\alphaout)$.}
\end{figure}

\subsubsection{Density of the angular
  measure}\label{subsubsec:angmeas}
A traditional way to describe the asymptotic dependence structure of a
standardized heavy tailed vector is by using the angular measure.
We transform the standardized vector $(I^a,O)\mapsto 
\bigl(\arctan(O/I^a),\sqrt{O^2+I^{2a}}\bigr)$ to polar coordinates and
then  the distribution of 
$\arctan(O/I^a)$ given $O^2+I^{2a}>m$, converges as $m\to\infty$ to
the 
distribution to a random variable $\Theta$. The distribution of
$\Theta$ is called the angular measure.
The density of
$\Theta$ can be calculated from Theorem \ref{t:two.measures}
in a similar fashion as in Corollory \ref{c:ratio} and
is given by 
\begin{eqnarray*}
f_\Theta(\theta) &\propto&
\frac{\gamma}{\deltain}(\cos\theta)^{\frac{\deltain}{a}+\frac{1}{a}-1}(\sin\theta)^{\deltaout -1}\int_0^\infty t^{c_1^{-1}+\deltain+a\deltaout } e^{-t(\cos\theta)^\frac{1}{a}-t^a\sin\theta}\, dt\, \\
&& +\, \frac{\alpha}{\deltaout }(\cos\theta)^{\frac{\deltain}{a}-1}(\sin\theta)^{\deltaout }\int_0^\infty t^{a-1+c_1^{-1}+\deltain+a\deltaout } e^{-t(\cos\theta)^\frac{1}{a}-t^a\sin\theta}\, dt\,. \\
\end{eqnarray*}

Two density approximations for the spectral density using network simulation and numerical iteration of the
$p_{ij}$ are obtained in  the same way as in Section
\ref{subsubsec:R}.
Using the same sets of parameters values as in
Figure 2, we overlay the density approximations with the theoretical density in Figure 3. 
The
truncation level was the $99.95\%$ percentile of $O^2+I^{2a}$.  The
agreement between the theoretical and estimated densities is quite
good across the range of parameter values used.  

The main difference
between Figures 2 and 3 is the choice of conditioning set.  In the first,
$I^a$ was conditioned to be large, while in the second the sum of
squares of the in- and out-degrees ($I^{2a}+O^2$) was conditioned to
be large.  Since the latter conditioning set is bigger and allows for
the case that the in-degree is  small relative to the out-degree, the
density function in a neighborhood 0 will have less weight in Figure 3
than Figure 2.

\begin{figure}[t]
	\centering
	\includegraphics[width=5in,height=5in]{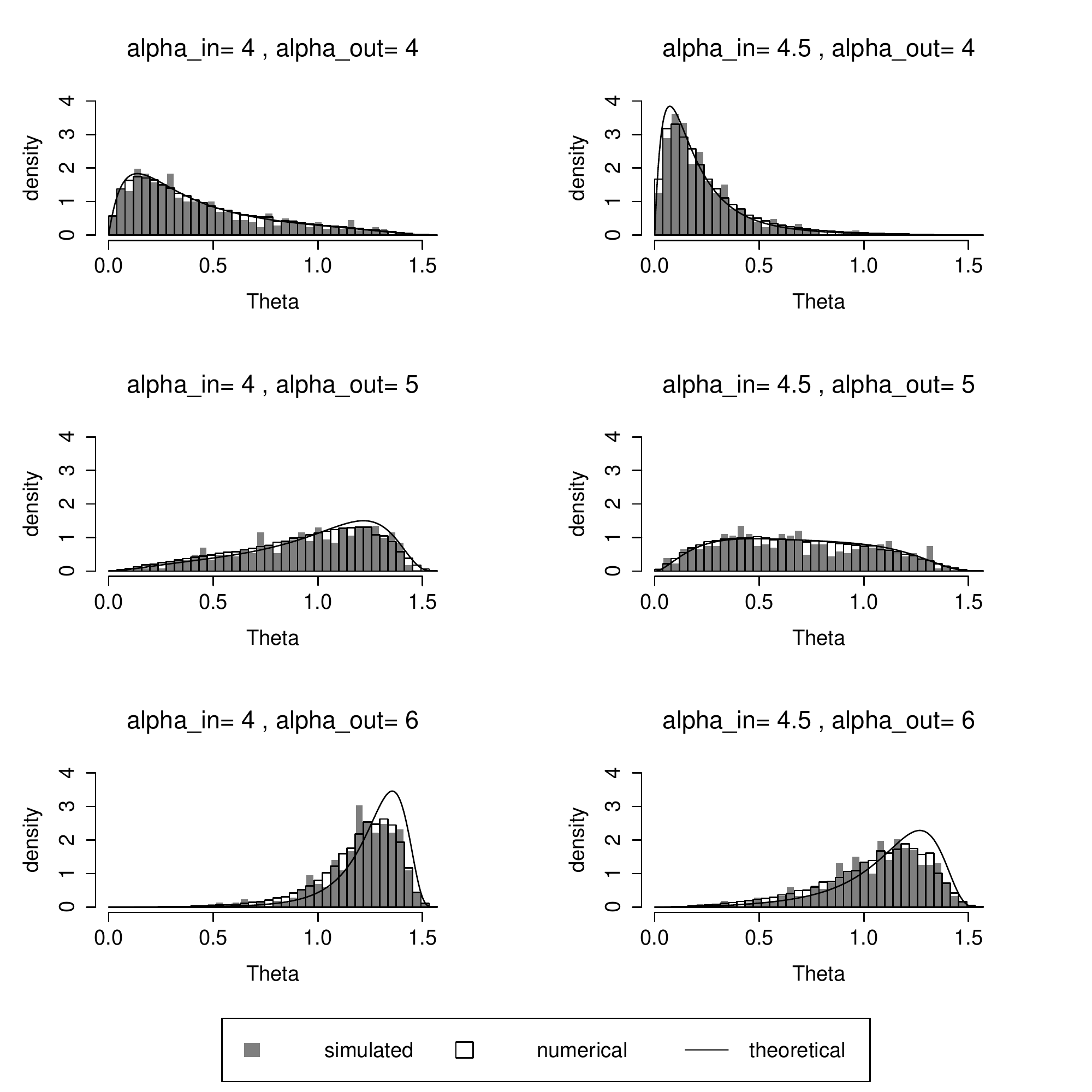}
	\caption{Comparison of the true angular density with estimates for   various values of $(\alphain,\alphaout)$.}
\end{figure}

\bibliographystyle{Genamystyle}
\bibliography{Genabibfile}

\begin{thebibliography}{7}
\expandafter\ifx\csname natexlab\endcsname\relax\def\natexlab#1{#1}\fi

\bibitem[Bollob\'as et~al.(2003)Bollob\'as, Borgs, Chayes and
  Riordan]{bollobas:borgs:chayes:riordan:2003}
{\sc B.~Bollob\'as, C.~Borgs, J.~Chayes {\rm and} O.~Riordan} (2003): Directed
  scale-free graphs.
\newblock In {\em Proceedings of the Fourteenth Annual ACM-SIAM Symposium on
  Discrete Algorithms (Baltimore, 2003)\/}. ACM, New York, pp. 132--139.

\bibitem[Jones(1971)]{jones:1971}
{\sc F.~Jones} (1971): {\em Partial Differential Equations\/}.
\newblock Springer-Verlag, New York.

\bibitem[Krapivsky and Redner(2001)]{krapivsky:redner:2001}
{\sc P.~Krapivsky {\rm and} S.~Redner} (2001): Organization of growing random
  networks.
\newblock {\em Physical Review E\/} 63:066123:1--14.

\bibitem[{Lindskog} et~al.(2013){Lindskog}, {Resnick} and
  {Roy}]{lindskog:resnick:roy:2013}
{\sc F.~{Lindskog}, S.~{Resnick} {\rm and} J.~{Roy}} (2013): Regularly Varying
  Measures on Metric Spaces: Hidden Regular Variation and Hidden Jumps.
\newblock Technical report, School of ORIE, Cornell University.
\newblock Preprint. Available at: {\tt {http://arxiv.org/abs/1307.5803}}.

\bibitem[Lo{\`e}ve(1977)]{loeve:1977}
{\sc M.~Lo{\`e}ve} (1977): {\em Probability Theory\/}, volume~1.
\newblock Springer-Verlag, New York.

\bibitem[Maulik et~al.(2002)Maulik, Resnick and
  Rootzen]{maulik:resnick:rootzen:2002}
{\sc K.~Maulik, S.~Resnick {\rm and} H.~Rootzen} (2002): Asymptotic
  independence and a network traffic model.
\newblock {\em Journal of Applied Probability\/} 39:671--699.

\bibitem[Resnick(2007)]{resnick:2007}
{\sc S.~Resnick} (2007): {\em Heavy-Tail Phenomena: Probabilistic and
  Statistical Modeling\/}.
\newblock Springer, New York.

\end{thebibliography}


\end{document}